\documentclass[final]{siamltex}

\usepackage{times}

\usepackage{amsfonts}
\usepackage{amsmath}
\usepackage{amssymb}
\usepackage{array}
\usepackage{subfigure}
\usepackage{mathdots}

\usepackage{algorithm}%,algorithmic}
\usepackage{algpseudocode}
\floatname{algorithm}{Algorithm}

\usepackage{graphicx}
\usepackage{color}
\usepackage{pgf,pgfarrows,pgfnodes}
\usepackage{scrpage2}
\usepackage{multirow}
\usepackage{listings} %Quellcode einbinden
\usepackage{tikz}

\newcommand{\N}{\ensuremath{\mathbb{N}}}

\newcommand{\R}{\ensuremath{\mathbb{R}}}
\newcommand{\C}{\ensuremath{\mathbb{C}}}

\newcommand{\ii}{\textnormal{i}}
\newcommand{\e}{\textnormal{e}}

\newcommand{\eip}[1]{\textnormal{e}^{2\pi\ii{#1}}}

\newcommand{\zb}[1]{\ensuremath{\boldsymbol{#1}}}
\newcommand{\pmat}[1]{\begin{pmatrix} #1 \end{pmatrix}}
\newcommand{\IpolI}{\mathcal{I}}

\DeclareMathOperator*{\diam}{diam}
\DeclareMathOperator*{\dist}{dist}

 \newtheorem{example}[theorem]{Example}
\numberwithin{equation}{section}
\numberwithin{table}{section}
\numberwithin{figure}{section}

\long\def\symbolfootnote[#1]#2{\begingroup%
\def\thefootnote{\fnsymbol{footnote}}\footnote[#1]{#2}\endgroup}
\newcommand{\dd}{\mathrm{d}}

\renewcommand{\mathbf}[1]{\ensuremath{\boldsymbol{#1}}}
\renewcommand{\textbf}[1]{{\ensuremath{\boldsymbol{#1}}}}

\renewcommand{\thefootnote}{\fnsymbol{footnote}}

\allowdisplaybreaks
\title{Fast evaluation of real and complex exponential sums}

\date{\today}
\date{}

\author{Stefan Kunis\footnotemark[2] \and Ines Melzer\footnotemark[2]}

\begin{document}

\maketitle

\begin{abstract}
Recently, the butterfly approximation scheme and hierarchical approximations have
been proposed for the efficient computation of integral transforms with oscillatory
and with asymptotically smooth kernels.
Combining both approaches, we propose a certain fast Fourier-Laplace transform, which in
particular allows for the fast evaluation of polynomials at nodes in the complex unit disk.
All theoretical results are illustrated by numerical experiments.
\end{abstract}

\begin{keywords}
trigonometric approximation, nonharmonic Fourier series, fast Fourier transform,
integral transforms, hierarchical matrices.
\end{keywords}

\begin{AMS}
%65T40, % Trigonometric approximation and interpolation
65T50, % Discrete and fast Fourier transforms
42A15, % Trigonometric interpolation
30E10,  % Approximation in the complex domain
65D05, % Interpolation
65F30 % other matrix algorithms
\end{AMS}

\footnotetext[2]{
  University Osnabr\"uck, Institute of Mathematics
  \{stefan.kunis,ines.melzer\}@math.uos.de
}

%%%%%%%%%%%%%%%%%%%%%%%%%%%%%%%%%%%%%%%%%%%%%%%%%%%%%%%%%%%%%%%%%%%%%%%%%%%%%%
\section{Introduction}
The fast Fourier transform (FFT) \cite{CoTu65,DuVe90,FFTW05} belongs to the
algorithms with large impact on science and engineering.
Generalizations have been given for nonequispaced nodes, see \cite{MiBo96,AyChSoCu03,Yi09,CaDeYi09,KuMe12,PoDeMaYi14}
for the recently suggested butterfly schemes and \cite{DuRo95,Bey95,St98,ElSt98,KeKuPo09} together with its references for some wider survey on
gridding type approximations.
Moreover, structured low rank approximations for integral transforms with smooth kernels have been
developed as fast multipole methods \cite{GrRo87,SuPi01,YiBiZo04} and hierarchical matrices \cite{Ha99,Be00,GrHa03,Be08,Ha09}.
One particular instance of a structured low rank approximation for a smooth kernel is given in \cite{Ro88} for a discrete Laplace
transform.
In all cases, the concept in such schemes is to trade exactness for efficiency; instead of precise computations
up to machine precision, the proposed methods guarantee a given target accuracy.
Neglecting logarithmic factors in the problem size and the target accuracy, the computational complexity of all these algorithms scales linear in the
problems size.

Discrete Laplace transforms have been developed in \cite{Ro88,St92} based on polynomial interpolation and on approximations with Laguerre polynomials, respectively.
In the first part of this paper, we present a matrix form of \cite{Ro88} and develop a generalization to more general kernel functions and a small improvement in the error estimate.
The main contribution of the paper is a combination of the discrete Laplace transform and a generalized fast Fourier transform (FFT), where we use the decomposition of the Laplace
transform explicitly and a small number of generalized FFTs as black box.
In particular, this allows for the fast evaluation of a polynomial, given by its monomial coefficients, at many nodes in the complex unit disk.
Alternatively, we might interpret this as an FFT with nonequispaced nodes in the upper half plane.
For notational convenience, all ideas are presented for one space dimension but can be generalized in a straightforward manner to the multivariate case.
Finally, the theoretical results on accuracy and computational complexity are illustrated by some numerical experiments.

%%%%%%%%%%%%%%%%%%%%%%%%%%%%%%%%%%%%%%%%%%%%%%%%%%%%%%%%%%%%%%%%%%%%%%%%%%%%%%
\section{Preliminaries}\label{sect:pre}
Let $q\in\N$ and the nodes $t_j=\cos\frac{2j+1}{2q}\pi$, $j=0,\hdots,q-1$, be the zeros of the $q$-th Chebyshev polynomial of the first kind.
Moreover, let $A:=[a,b]$, $a<b$, be an interval with diameter $\diam A:=b-a$, midpoint $c^A:=\frac{a+b}{2}$, and
Chebyshev nodes $y_j^A:=c^A + \frac{\diam A}{2} t_j$, $j=0,\dots,q-1$.
The corresponding Lagrange polynomials $L_r^{A}: A\rightarrow\R$, $r=0,\hdots,q-1$, are
\begin{equation*}
 L_r^{A}(y):= \prod_{\substack{j=0 \\ j \neq r}} ^{q-1} \frac{y-y_j^A}{y_r^A-y_j^A}
 =\frac{\frac{\lambda_r}{y -y_r^A}}{\sum_{s=0}^{q-1} \frac{\lambda_s}{y-y_s^A}},\quad
 \lambda_r:=(-1)^r \sin \left(\frac{2r+1}{2q} \pi\right),
\end{equation*}
where the second identity is called barycentric formula and allows for a stable evaluation, cf.~\cite{BeTr04}.
We define the interpolation operator $\mathcal{I}_q^{A}:C(A) \rightarrow C(A)$,
\begin{equation*}
 \mathcal{I}_q^A g := \sum_{j=0}^{q-1} g(y_j^A) L^A_j,
\end{equation*}
which fulfills
\begin{align}
  \label{eq:Iq}
  \|g-\mathcal{I}_q^A g\|_{C(A)} &\le \frac{\diam(A)^q}{2^{2q-1} q!} \|g^{(q)}\|_{C(A)},\quad \text{if}\; g\in C^{(q)}(A),\\
  \label{eq:Iq2}
  \|\mathcal{I}_q^A\|&:=\sup_{\|g\|_{C(A)}=1} \|\mathcal{I}_q^A g\|_{C(A)}\le 1+\frac{2}{\pi}\log q.
\end{align}

Now let $A,B\subset\R$ be two intervals and $\kappa: A \times B \rightarrow \R$, then we define the interpolation in both variables by
\begin{equation*}
 \mathcal{I}_q^{A\times B}:= \mathcal{I}_q^{A} \otimes \mathcal{I}_q^{B},\quad
 \mathcal{I}_q^{A\times B} \kappa (y,\xi)  : = \sum_{s=0}^{q-1} \sum_{r=0}^{q-1} L_s^A(y) \kappa(y_s^A,\xi_r^B) L_r^B(\xi),
\end{equation*}
where $\xi_r^B$ denote the Chebyshev nodes in the interval $B$.
We note in passing that $\mathcal{I}_q^{A} \otimes \mathcal{I}_q^{B}=
(\mathcal{I}_q^A \otimes \mathcal{I})(\mathcal{I}\otimes \mathcal{I}_q^B)$, where $\mathcal{I}$ denotes the identity, and
$\|\mathcal{I}_q^A \otimes \mathcal{I}\|\le 1+\frac{2}{\pi}\log q$.

%%%%%%%%%%%%%%%%%%%%%%%%%%%%%%%%%%%%%%%%%%%%%%%%%%%%%%%%%%%%%%%%%%%%%%%%%%%%%%
\section{Rokhlin's discrete Laplace transform and a generalization}\label{sect:lt}
We generalize and slightly improve \cite{Ro88} to a method computing
\begin{equation}\label{eq:sum}
 f(y_i)=\sum_{j=1}^{N} \hat f_j \kappa(y_i,\xi_j),\quad i=1,\hdots,N,
\end{equation}
for specific kernels $\kappa$, given $N\in\N$, $y_1>\hdots>y_N>0$, $\xi_1>\hdots>\xi_N>0$, and $\hat f_k\in\C$.

Adopting the terms from the hierarchical matrices literature, see e.g.~\cite{Be08,Ha09}, 
a function \mbox{$\kappa:(0,\infty)\times (0,\infty)\rightarrow\R$} is said to be \emph{asymptotically smooth} if
there exist constants $C,\mu,s\ge 0$, $\nu\in\R$ such that for all $q\in\N$ the conditions
\begin{equation*}
  \left|y^q \partial^q_y \kappa(y,\xi)\right| \le C q! \mu^q q^{\nu} (y\xi)^{-s} 
 \quad \text{and}\quad
 \left|\xi^q \partial^q_\xi \kappa(y,\xi)\right| \le C q! \mu^q q^{\nu} (y\xi)^{-s}
\end{equation*}
are fulfilled for all $y,\xi \in (0,\infty)$. The parameter $s$ characterizes the singularity of the kernel for $y\xi=0$.
Moreover, we call two intervals $A,B\subset [0,\infty)$ \emph{admissible} if
\begin{equation*}
 \diam(A)\le\dist(A,0)
 \quad \text{and}\quad
 \diam(B)\le\dist(B,0).
\end{equation*}

\begin{theorem}\label{thm:LapLocalError}
 Let $q\in\N$, $q\ge 2$, $A,B\subset(0,\infty)$ be admissible, and $\kappa\colon A\times 
B\rightarrow\R$ be 
asymptotically smooth with constants $C,\mu, s\ge 0$ and $\nu\in \R$, then we have
 \begin{equation*}
  \left\|\kappa - \IpolI_q^{A\times B}\kappa\right\|_{C(A\times B)} \le \frac{C \mu^q 
	  q^{\nu}}{2^{2q-1}} \left(2+\frac{2}{\pi}\log q\right) \left(\dist(A,0)\dist(B,0) \right)^{-s}.	
 \end{equation*}
\end{theorem}

\begin{proof}
 For fixed $\xi\in B$ and $g:A \rightarrow \R$, $g(y):=\kappa(y,\xi)$, we apply the error formula \eqref{eq:Iq} and obtain
 \begin{equation*}
	 \|g-\IpolI_q^A g\|_{C(A)} \le \frac{\diam(A)^q}{2^{2q-1} q!} \| g^{(q)}\|_{C(A)} .
 \end{equation*}
 The asymptotic smoothness and the admissibility implies
 \begin{equation*}
	 \left| g^{(q)}(y) \right| \le  C q! \mu^q q^\nu y^{-q} (y\xi)^{-s}\le C q!  \mu^q q^\nu (\dist(A,0))^{-q} 
(y\xi)^{-s}
\end{equation*}
and in conclusion
 \begin{equation*}
	 \|g-\IpolI_q^A g\|_{C(A)} \le C \mu^q q^\nu 2^{1-2q} \sup_{y\in A, \xi \in B}|(y\xi)^{-s}|.
 \end{equation*}
The same estimate holds true with respect to $\xi\in B$. From this and together with the bound \eqref{eq:Iq2}, we conclude
 \begin{equation*}
	 \begin{split}
  \left\|\kappa - \IpolI_q^{A\times B}\kappa\right\|_{C(A\times B)}
  &\le
  \left\|\kappa - \left(\IpolI_q^{A}\otimes\IpolI\right)\kappa\right\|_{C(A\times B)}
  +  \left\|\IpolI_q^{A}\otimes\IpolI\right\| 
  \left\|\kappa - \left(\IpolI\otimes\IpolI_q^{B}\right)\kappa\right\|_{C(A\times B)}\\
  &\le  
  \frac{C \mu^q q^{\nu}}{2^{2q-1}} \left(2+\frac{2}{\pi}\log q\right) \sup_{y\in A, \xi \in 
B}|(y\xi)^{-s}|. 
	\end{split}
 \end{equation*}
The conditions $y\xi \ge \dist(A,0)\dist(B,0)$ and $s\ge 0$ imply the assertion. 
\qquad \end{proof}

We start with the discrete Laplace transform from \cite{Ro88}, i.e., $\kappa(y,\xi)=\e^{-\xi y}$. In matrix form, the computation of \eqref{eq:sum} reads as
\begin{equation}\label{eq:MatK}
 \mathbf{f}:=\mathbf{K \hat f} \quad \text{with} \quad \mathbf{K} :=\left(\kappa(y_j,\xi_k) \right)_{j,k=1,\dots,N} \quad \text{and} \quad \mathbf{\hat f}:=\left(\hat f_j\right)_{j=1,\hdots,N}. 
\end{equation}
Algorithm \ref{alg:FLT} finally computes an approximation $\mathbf{\tilde f} \approx \mathbf{f}$, precisely defined in Equation \eqref{eq:tildef} and we refer the reader to Theorem \ref{thm:approxE} for a shortened and slightly improved error estimate \cite{Ro88}.
Subsequently, we focus on the matrix partitioning of $\zb K$ into blocks of approximate low rank and the derivation of the computational costs of Algorithm \ref{alg:FLT}.
This will allow for the generalization to other kernels in the end of this section and for an application when evaluating polynomials in the unit disk in Section \ref{sect:exp}.

\begin{definition}\label{def:M}
For given target accuracy $\varepsilon>0$ and interval lengths $y_1,\xi_1>0$, we define
\begin{align*}
	q 	  & :=\lceil\frac{1}{2}+\log_4 1/\varepsilon\rceil,
        &M        & := \left\lceil\log_2\frac{y_1\xi_1}{\varepsilon}\right\rceil+1,\\
        \ell_m    & := \max(1,\lfloor \log_2(y_1 \xi_1) - m - \log_2(\log 1/\varepsilon)\rfloor + 1),
        &L_m       & := M-m,
 \end{align*}  
 for $m=1,\hdots,M-1$ and set up the geometric partitioning, see Figure \ref{fig:decomposition} for an illustration,
 \begin{align*}
	Y         & := [0,y_1],\quad
	&Y_{M}    & := \left[0, \frac{y_1}{2^{M-1}}\right],\quad
	&Y_m      & := \left(\frac{y_1}{2^m},\frac{y_1}{2^{m-1}} \right],\\
	\Omega    & := [0,\xi_1], \quad
	&\Omega_M & := \left[0, \frac{\xi_1}{2^{M-1}} \right],\quad
	&\Omega_m & := \left(\frac{\xi_1}{2^m},\frac{\xi_1}{2^{m-1}} \right].
 \end{align*}
For ease of notation in Algorithm \ref{alg:FLT}, we moreover define
\begin{align*}
 \mathbf{L}^{\Omega_l} &:= \left(L_r^{\Omega_l} (\xi_j) \right)_{\xi_j \in \Omega_l,r=0,\dots,q-1},
 \quad
 &\mathbf{L}^{Y_m} &:= \left(L_s^{Y_m} (y_i) \right)_{y_i \in Y_m,s=0,\dots,q-1}, \\
 \mathbf{\hat f}^{\Omega_l} & := \left( \hat f_j\right)_{\xi_j \in \Omega_l},
 \quad
 &\mathbf{K}^{Y_m,\Omega_l} &:= \left( \kappa(y_s^{Y_m},\xi_r^{\Omega_l}) \right)_{s,r=0}^{q-1,q-1},
\end{align*}
for $m,l=1,\dots,M-1$.
\end{definition}

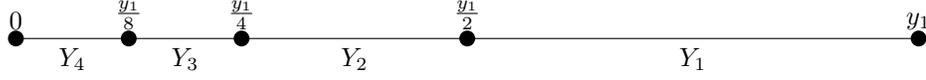
\begin{figure}[H]
	\centering
	\begin{tikzpicture}[scale=1.5]

		\coordinate[label=above:$0$] (00) at (0,0);
		\coordinate[label=above:$\frac{y_1}{8}$] (01) at (1,0);
		\coordinate[label=above:$\frac{y_1}{4}$] (02) at (2,0);
		\coordinate[label=above:$\frac{y_1}{2}$] (04) at (4,0);
		\coordinate[label=above:$y_1$] (08) at (8,0);
		
		\draw (00) -- (01) node [midway,below] {$Y_{4}$};
		\draw (01) -- (02) node [midway,below] {$Y_{3}$};
		\draw (02) -- (04) node [midway,below] {$Y_{2}$};
		\draw (04) -- (08) node [midway,below] {$Y_{1}$};
		\fill (00) circle (2pt);	
		\fill (01) circle (2pt);
		\fill (02) circle (2pt);	
		\fill (04) circle (2pt);
		\fill (08) circle (2pt);	
		
	\end{tikzpicture}
	\caption{ Decomposition of $Y=[0,y_1]$ for $M=4$.}
	\label{fig:decomposition}
\end{figure}

Regarding the computational costs,  we first note that the computation of $g_m$ in Algorithm~\ref{alg:FLT} by means of a cumulative summation takes $\mathcal{O}(N)$ operations.
The spatial partitions in Definition~\ref{def:M} yield a partition of the matrix $\mathbf{K}$ into admissible blocks.
Considering only the case when the kernel is approximated by interpolation, Algorithm \ref{alg:FLT} factors out row- and column bases by
 \begin{align*}
  \zb {\tilde K}
       &=\left(\zb L^{Y_m}\zb K^{Y_m,\Omega_{\ell}}\left(\zb L^{\Omega_{\ell}}\right)^{\top}\right)_{m=1,\ell=\ell_m}^{M-1,L_m}\\
       &=\diag(\zb L^{Y_1},\hdots,\zb L^{Y_{M-1}})
        \pmat{\zb K^{Y_m,\Omega_{\ell}}}_{m=1,\ell=\ell_m}^{M-1,L_m}
        \diag(\zb L^{\Omega_{\ell_m}},\hdots,\zb L^{\Omega_{L_m}})^{\top}
       %&=\pmat{\zb L^{Y_1} & 0 & 0 \\ 0 & \ddots & 0 \\ 0 & 0 & \zb L^{Y_{M-1}}}
       % \pmat{\ddots & \vdots & \iddots\\ \cdots & \zb K^{Y_m,\Omega_{\ell}} & \cdots \\ \iddots & \vdots & \ddots}_{m=1,\ell=\ell_m}^{M-1,L_m}
       % \pmat{\left(\zb L^{\Omega_{\ell_m}}\right)^{\top} & 0 & 0 \\ 0 & \ddots & 0 \\ 0 & 0 & \left(\zb L^{\Omega_{L_m}}\right)^{\top}}.
 \end{align*}
 Clearly, the application of the rightmost block diagonal matrix takes at most $\sum_{\ell=\ell_m}^{L_m} q|\Omega_\ell| \le q N$ operations.
 The second matrix has at most $L_m - \ell_m \le 2 \log(1/\varepsilon)$ blocks in its $m$-th block row, in total its application takes
 $\mathcal{O}(Mq^2\log\frac{1}{\varepsilon})$ operations and the multiplication with the left block diagonal matrix again takes at most $q N$ operations.
 Finally note that we neglect the precomputation of the necessary matrices $\mathbf{L}^{\Omega_\ell}$, $\mathbf{L}^{Y_m}$, and $\mathbf{K}^{Y_m,\Omega_\ell}$.
 In total, Algorithm \ref{alg:FLT} takes
 \begin{equation*}
   \mathcal{O}\left(N\log\frac{1}{\varepsilon}+\log^3\frac{1}{\varepsilon}\log\frac{y_1\xi_1}{\varepsilon}\right)
 \end{equation*}
 floating point operations.

\begin{algorithm}[ht]
  \caption{Laplace transform}
	\label{alg:FLT}
	\begin{algorithmic}
		\Require\ \\
		$\varepsilon \in (0,1)$  \Comment{target accuracy}\\
		$N \in \N$ 	\Comment{number of sampling nodes}\\
		$\xi_1 > \xi_2 \dots > \xi_N >0$ \Comment{nodes in frequency domain}\\
		$y_1 >y_2 > \dots > y_N >0$ \Comment{nodes in spatial domain} \\
		$\zb {\hat f}\in\C^N$ \Comment{Fourier coefficients}
		\Ensure\ \\
		$\mathbf{\tilde f}\in\C^N$, $\mathbf{\tilde f}\approx \zb K \zb {\hat f}$ \Comment{samples in spatial domain}
		\vspace{1ex}
		\hrule
		\vspace{1ex}
		%\State $M=\left\lceil \log_2 \frac{y_1 \xi_1}{\varepsilon} \right\rceil +1 $ \Comment{number of decompositions}
		%\State $q=\lceil\frac{1}{2}+\log_4 1/\varepsilon\rceil$	 \Comment{approximation rank}
		\vspace{1ex}
		\State{$g_M=\sum_{\xi_j \in \Omega_M} \hat f_j$}
		\For{$m=M-1,\hdots,1$}
			\State{$g_{m}=g_{m+1} + \sum_{\xi_j \in \Omega_{m}} \hat f_j$}
		\EndFor
		\State $\zb {\tilde{f}}^{Y_M}= g_1$
		\For{$\ell=1,\dots, M-1$}
			\State $\mathbf{v}^{\Omega_\ell}= \left(\mathbf{L}^{\Omega_\ell}\right)^{\top} \mathbf{\hat f}^{\Omega_\ell}$
		\EndFor
			\vspace{1ex}
		\For{$m=1,\dots,M-1$}
			\State $\mathbf{h}^{Y_m} = \sum_{\ell=\ell_m}^{L_m} \mathbf{K}^{Y_m,\Omega_\ell} \mathbf{v}^{\Omega_\ell}$
			\State $\zb {\tilde{f}}^{Y_m}=\mathbf{L}^{Y_m} \mathbf{h}^{Y_m}+g_{L_m+1}\zb 1$
		\EndFor
		\end{algorithmic}

	\end{algorithm}

%%%%%%%%%%%%%%%%%%%%%%%%%%%%%%%%%%%%%%%%%%%%%%%%%%%%%%%%%%%%%%%%
This discrete Laplace transform can be generalized to asymptotically smooth kernels. 
The slightly modified versions of Algorithm~\ref{alg:FLT} and Theorem~\ref{thm:approxE} read as follows. 

\begin{theorem}
 Let $\kappa:(0,\infty)\times(0,\infty) \to \R$ asymptotically smooth with constants $\nu\in\R$, $s\ge 0$ and $\mu\in (0,4)$.
 Furthermore, let $\varepsilon >0$, $y_1,\xi_1>0$ be fixed for all $N\in \N$, and the nodes $y_1>y_2>\hdots>y_N>0$, $\xi_1>\xi_2>\hdots>\xi_N>0$ be quasi-uniformly such that
 the intervals $[0,y_1/N]$ and $[0,\xi_1/N]$ contain only a constant number of nodes for all $N\in \N$, respectively. 
 We set $M:=\left\lceil \log_2 \frac{y_1 \xi_1 N}{\varepsilon} \right\rceil +1$ and $q \in \mathcal{O}(\frac{\log N}{\varepsilon})$ and define
 the approximation $\tilde f: Y \to \C$,
\begin{equation*}
	\tilde f(y): = \begin{cases}
		f(y) & y \in Y_M,\\
    \sum_{\ell=1}^{M-1} \sum_{\xi_j \in \Omega_\ell} \hat f_j \IpolI^{Y_m \times \Omega_\ell} \kappa(y,\xi_j)
    + \sum_{\xi_k \in \Omega_M} \hat f_k \kappa(y,\xi_k) & y\in Y_m,\; 1\le m<M.
       \end{cases}
\end{equation*}
Then the error estimate
$%\begin{equation*}
  \|\mathbf{f} - \mathbf{\tilde f} \|_\infty \le \varepsilon \| \mathbf{\hat f}\|_1
$%\end{equation*}
is fulfilled and the modified Algorithm \ref{alg:FLT} computes this approximation in
$%\begin{equation*}
	\mathcal{O}(N\log \frac{N}{\varepsilon}+\log^4\frac{N}{\varepsilon})
$%\end{equation*}
 floating point operations.
\end{theorem}

\begin{proof}
  For $\mu \in (0,4)$, $C<0$, and $\nu \in \R$, there exist constants $\tilde C>0$ and $\frac{\mu}{4}< c < 1$, such that
  \begin{equation*}
		2C \left( \frac{\mu}{4} \right)^q q^\mu \left(2+\frac{2}{\pi}\log q\right) \le \tilde C c^q
  \end{equation*}
  holds true for all $q \in \N$. Theorem~\ref{thm:LapLocalError} implies for $A,B \subset(0,\infty)$ a local approximation 
\begin{equation}\label{eq:locApprox}
	\| \kappa - \IpolI_q^{A \times B} \kappa \|_{C(A\times B)} \leq \tilde C c^q (\dist(A,0)\dist(B,0))^{-s}.
\end{equation}
Let now $A:=Y_m$ and $B:= \Omega_\ell$ with $m,\ell \neq M$ be given.
Since $M-1 \le \log_2 \frac{y_1 \xi_1 N}{\varepsilon}+\frac{1}{2}$, it follows
\begin{equation}\label{eq:distAB}
 \dist(A,0)\dist(B,0) \ge \frac{y_1}{2^{M-1}} \frac{\xi_1}{2^{M-1}} \ge  \frac{\varepsilon^2}{ 2y_1 
\xi_1 N^2}.
 \end{equation}
 Applying $q \ge \frac{1}{|\log c|} \log\left (\frac{ \tilde C \varepsilon^{2s+1}}{(2y_1\xi_1)^s N^{2s})} \right)$ implies
  \begin{align*}
  \left\|\kappa - \IpolI_q^{A\times B}\kappa\right\|_{C(A\times B)} 	& \le \tilde C (2y_1 \xi_1)^s \left(\frac{N}{\varepsilon}\right)^{2s} c^q \le \varepsilon.
 \end{align*}
 
 Finally, $M \in \mathcal{O}(\log\frac{N}{\varepsilon})$ leads to a constant number $(\mathcal{O}(\varepsilon))$ of nodes in the near fields $Y_M$ and $\Omega_M$ where we apply direct computations.
 Of course the approximations of the kernel by either zero or one as in Lemma \ref{lem:approxE} cannot be done in general and thus we
 let run $\ell=1,\hdots,M-1$ in Algorithm \ref{alg:FLT}. % Lemma \ref{lem:approxE}iii).
 Since $q \in \mathcal{O}(\log \frac{N}{\varepsilon})$, Algorithm \ref{alg:FLT} now takes $\mathcal{O}(N\log \frac{N}{\varepsilon}+\log^4\frac{N}{\varepsilon})$
 floating point operations.
\qquad \end{proof}

\begin{example}\label{lem:smoothbessel}
 Let the modified Bessel function of the second kind $K_\eta \colon (0,\infty) \to \R$, 
 \begin{equation*}
  K_\eta(x) : = \int_0^{\infty} \e^{-x \cosh(t)} \cosh(\eta t) \dd t,
 \end{equation*}
 be given, set $\eta=\frac{1}{2}$, and consider the kernel $\kappa: (0,\infty) \times (0,\infty) \rightarrow \R$,
 \begin{equation*}
  \kappa(y,\xi):=K_{\frac{1}{2}}(y \xi)=\sqrt{\frac{\pi}{2y\xi}}\e^{-y\xi}
 \end{equation*}
 which possesses a singularity for $y\xi=0$.
 Induction on $q\in \N_0$ shows the identity 
 \begin{equation*}
  \frac{\partial^q}{\partial y^q} \kappa(y,\xi) = \sqrt{\frac{\pi}{2}} (-\xi)^q \e^{-y\xi} 
	\sum_{k=0}^q \binom{q}{k}\frac{\prod_{j=0}^{k-1}(2j+1)}{2^{k}}(y\xi)^{-(2k+1)/2},
 \end{equation*}
 from which the asymptotic smoothness
 \begin{equation*}
	\left|\frac{\partial^q}{\partial y^q} \kappa(y,\xi)\right|
	 \le \sqrt{\frac{\pi}{2y\xi}} \xi^q \e^{-y\xi} 	\sum_{k=0}^q \binom{q}{k}k!(y\xi)^{-k}
	 \le  \sqrt{\frac{\pi}{2}} q! (y\xi)^{-\frac{1}{2}} y^{-q}
 \end{equation*}
 with constants $C=\sqrt{\frac{\pi}{2}}$, $\mu=1$, $\nu=0$, and $s=\frac{1}{2}$ follows.
 Theorem~\ref{thm:LapLocalError} implies the corresponding local error estimate and thus $\tilde C= 2\pi$ and $c=\frac{1}{3}$ can be chosen in \eqref{eq:locApprox}.
\end{example}

%%%%%%%%%%%%%%%%%%%%%%%%%%%%%%%%%%%%%%%%%%%%%%%%%%%%%%%%%%%%%%%%%%%%%%%%%%%%%%
\section{Evaluation of polynomials in the unit disk}\label{sect:exp}
We are interested in evaluating the generalized polynomials $f:\C\rightarrow\C$,
\begin{equation}\label{eq:fz}
 f(z)=\sum_{k=1}^{N} \hat f_k z^{\xi_k}
\end{equation}
at nodes $z_j\in Z:=\{z\in\C: |z|\le 1\}$, $j=1,\hdots,N$, and for exponents $\xi_k\in[1,N]$, $N\in\N$, where we
exclude the nonpositive real axis $z_j\le 0$ for noninteger exponents $\xi_k$.

The main idea now consists in a combination of the discrete Laplace and a nonequispaced Fourier transform \cite{DuRo95,Bey95,St98,ElSt98,KeKuPo09,Yi09,KuMe12}.
We write
\begin{equation}\label{eq:zpolar}
 z=\e^{-y} \eip{x},\quad y\in [0,\infty),\;x\in [0,1),
\end{equation}
and note that the summation \eqref{eq:fz} is a matrix vector multiplication with the matrix
\begin{equation*}
 \zb C:=\zb A\odot\zb K
\end{equation*}
where the Fourier matrix $\zb A$ is given by
\begin{equation}\label{eq:matA}
 \zb A:=(\eip{\xi_k x_j})_{j,k=1,\hdots,N}
\end{equation} 
and the Laplace matrix $\zb K$ is given by \eqref{eq:MatK}.
Of course, the symbol $\odot$ denotes the Hadamard (pointwise) product and we set $\|\zb M\|_{1\rightarrow\infty}:=\sup\{\|\zb M\zb x\|_{\infty}:\|\zb x\|_1=1\}$.
We have the following result when approximating the factors as in the previous sections.
\begin{lemma}\label{lem:errorAdotK}
 Let $\varepsilon\in(0,1)$ and the Fourier and the Laplace matrix be approximated by
 \begin{equation*}
  \|\zb A- \zb{\tilde A}\|_{1\rightarrow\infty} \le \frac{\varepsilon}{3},\quad
  \|\zb K- \zb{\tilde K}\|_{1\rightarrow\infty} \le \frac{\varepsilon}{3},
 \end{equation*}
 then
 \begin{equation*}
  \|\zb A\odot\zb K- \zb{\tilde A}\odot \zb{\tilde K}\|_{1\rightarrow\infty} \le \varepsilon.
 \end{equation*}
\end{lemma}
{\em Proof}.
 The estimate simply follows from $\|\zb A\|_{1\rightarrow\infty}=1$, $\|\zb {\tilde K}\|_{1\rightarrow\infty}\le 1+\frac{\varepsilon}{3}$, and
 \begin{equation*}
  \zb A\odot\zb K- \zb{\tilde A}\odot \zb{\tilde K}
  =
  \zb A\odot(\zb K-\zb{\tilde K})+(\zb A- \zb{\tilde A})\odot \zb{\tilde K}.\eqno\endproof
 \end{equation*}

We use the hierarchical decomposition of the discrete Laplace transform and realize
matrix vector products with matrix blocks by the following technique.
\begin{lemma}\label{lem:AdotK}
 By slight abuse of notation, let $\mathbf{K}=\mathbf{L}^Y\mathbf{K}^{Y,\Omega}\left(\mathbf{L}^{\Omega}\right)^\top$
 denote a single matrix block of the Laplace matrix and $\zb A$ the associated block of the Fourier matrix, then
 \begin{equation*}
  \left(\mathbf{A}\odot\mathbf{K}\right)\mathbf{\hat f}
  %=\left(\mathbf{L}^{Y}\odot \mathbf{A}\left(\mathbf{L}^{\Omega}\odot\left(\mathbf{1}^{\top}\otimes \mathbf{\hat f}\right)\right)
  %\left(\mathbf{K}^{Y,\Omega}\right)^\top\right)\mathbf{1},
  =\left(\mathbf{L}^{Y}\odot \mathbf{A} \left( \diag\mathbf{\hat f}\right)\mathbf{L}^{\Omega} \left(\mathbf{K}^{Y,\Omega}\right)^\top \right)\mathbf{1},
 \end{equation*}
 where $\mathbf{1}:=(1,\hdots,1)^{\top}\in\R^{q}$.
\end{lemma}
\begin{proof}
 The simplest case is $\mathbf{K}=\mathbf{l}^Y \left(\mathbf{l}^{\Omega}\right)^\top$ from which $a_{ij}\cdot k_{ij}=l^Y_i a_{ij} l^{\Omega}_j$
 and the assertion follows.
 This implies the result since
 \begin{equation*}
  \mathbf{K}=\mathbf{L}^Y \mathbf{K}^{Y,\Omega} \left(\mathbf{L}^{\Omega}\right)^\top
  =\sum_{r=1}^q \mathbf{l}^Y_r  \left(\mathbf{\tilde l}^{Y,\Omega}_r\right)^{\top},
 \end{equation*}
 where $\mathbf{l}^Y_r$ and $\mathbf{\tilde l}^{Y,\Omega}_r$ denote the columns of $\mathbf{L}^Y$ and $\mathbf{L}^{\Omega}\left(\mathbf{K}^{Y,\Omega}\right)^{\top}$,
 respectively.
\qquad \end{proof}

\begin{theorem}\label{thm:approxZ}
 Let $N\in\N$, $\varepsilon>0$, $z_j\in Z$, $j=1,\hdots,N$, $\mathbf{\hat f}\in\C^N$ and
 $\mathbf{f}:=\mathbf{C}\mathbf{\hat f}$ with
 \begin{equation*}
  \mathbf{C}=(z_j^k)_{j,k=1,\hdots,N}
 \end{equation*}
 be given. 
 Algorithm \ref{alg:XFLT} takes
 $%\begin{equation*}
   \mathcal{O}\left(N\log\frac{1}{\varepsilon} \log^2\frac{N}{\varepsilon}\right)
 $%\end{equation*}
 or $\mathcal{O}\left(N\log N \log\frac{1}{\varepsilon} \log^3\frac{N}{\varepsilon}\right)$ floating point operations using the
 FFT for nonequispaced nodes in time and frequency \cite{ElSt98} or the butterfly sparse Fourier transform \cite{KuMe12}, respectively.
 Its output $\zb{\tilde f}\in\C^N$ fulfills the error estimate
 $%\begin{equation*}
  \|\mathbf{f}-\mathbf{\tilde f}\|_{\infty} \le \varepsilon \|\mathbf{\hat f}\|_1.
 $%\end{equation*} 
\end{theorem}
 
\begin{proof}
 We start by noting that $|z_j|<\varepsilon$ implies
 $|f(z_j)|<\varepsilon \|\zb{\hat f}\|_1$ in \eqref{eq:fz}, i.e.,  $f(z_j)$ can be approximated by zero.
 We collect the associated nodes $Z^0:=\{z_j\in\C: |z_j|<\varepsilon\}$ and set the result $\zb 
{\tilde f}\in\C^N$,
 restricted to these nodes, to zero
 \begin{equation*}
  \zb {\tilde f}^{Z_0}:=\zb 0.
 \end{equation*}
 Now write all nodes in polar form \eqref{eq:zpolar}, where we can assume
 $0\le y_N\le \hdots \le y_1\le \log\frac{1}{\varepsilon}$ for the rest of the proof.
 
 The decomposition of the frequency nodes and the spatial nodes in dyadic intervals
 $\Omega_\ell$, $Y_m$, $\ell,m=1,\hdots,M$, cf.~Definition \ref{def:M}, induces
 a decomposition of the unit disk $Z$ into concentric bands and a decomposition of the nodes
 $\{x_j:j=1,\hdots,N\}=\dot\cup_{m=1,\hdots,M} X_m$, $X_m:=\{x_j:y_j\in Y_m\}$.
 Accordingly, we denote restrictions of the Fourier matrix, the Fourier coefficients, and the result vector by superscripts with these sets.
 In particular, spatial nodes are close to one in modulus $z_j\in Z_M$ if and only if $y_j\in Y_M$ and thus we set
 \begin{equation*}
  \zb {\tilde f}^{Z_M}:=\zb {\tilde A}^{X_M,[1,\xi_1]} \zb {\hat f}.
 \end{equation*}
 Here and subsequently, all multiplications with submatrices of $\zb A$ are realized by padding zeros to the input vector, multiplying with $\zb A$, and
 restricting to the desired results.
 
 Regarding the most interesting part of the approximation, the decomposition from Theorem \ref{thm:approxE} yields
 \begin{equation*}
  \zb A \odot \zb {\tilde K}
  =\left(\zb A^{X_m,\Omega_{\ell}}\odot\zb L^{Y_m}\zb K^{Y_m,\Omega_{\ell}}\left(\zb L^{\Omega_{\ell}}\right)^{\top}\right)_{m=1,\ell=\ell_m}^{M-1,L_m}.
 \end{equation*}
 For notational simplicity, we apply Lemma \ref{lem:AdotK} to one block row $m$ and two (artificial) block columns $\ell=0,1$, $\Omega:=\Omega_0\cup\Omega_1$, in
 \begin{align*}
  &\pmat{  \zb A^{X_m,\Omega_0}\odot\zb L^{Y_m}\zb K^{Y_m,\Omega_0}\left(\zb L^{\Omega_0}\right)^{\top}
        & \zb A^{X_m,\Omega_1}\odot\zb L^{Y_m}\zb K^{Y_m,\Omega_1}\left(\zb L^{\Omega_1}\right)^{\top}}
        \zb {\hat f}^{\Omega}\\
  &\quad=\left(\zb L^{Y_m}\odot \zb A^{X_m,\Omega_0} \left(\diag \zb{\hat f}^{\Omega_0}\right)\zb L^{\Omega_0}\left(\zb K^{Y_m,\Omega_0}\right)^{\top}\right)\zb 1\\
  &\quad+\left(\zb L^{Y_m}\odot \zb A^{X_m,\Omega_1} \left(\diag \zb{\hat f}^{\Omega_1}\right)\zb L^{\Omega_1}\left(\zb K^{Y_m,\Omega_1}\right)^{\top}\right)\zb 1\\
  &\quad=\left(\zb L^{Y_m}\odot \zb A^{X_m,\Omega}
	  \pmat{\left(\diag \zb{\hat f}^{\Omega_0}\right)\zb L^{\Omega_0}\left(\zb K^{Y_m,\Omega_0}\right)^{\top} \\
	        \left(\diag \zb{\hat f}^{\Omega_1}\right)\zb L^{\Omega_1}\left(\zb K^{Y_m,\Omega_1}\right)^{\top}} \right)\zb 1.
 \end{align*}
 Now, the error estimate is a straightforward consequence of Lemma \ref{lem:errorAdotK}.
 
 The complexity estimate follows when considering the dominant computation in the second last line $\zb F=\hdots$ in Algorithm \ref{alg:XFLT}.
 We have $M=\mathcal{O}(\log\frac{N}{\varepsilon})$ steps in the outer loop and $q=\mathcal{O}(\log\frac{1}{\varepsilon})$ right hand sides for the
 multiplication with the approximate Fourier matrix $\zb {\tilde A}$, which computational needs are given by $\mathcal{O}(N \log N \log^2 \frac{N}{\varepsilon})$ for the butterfly sparse Fourier transform \cite{KuMe12} or by $\mathcal{O}(N\log\frac{N}{\varepsilon})$ floating point operations the so-called fast Fourier transform for nonequispaced nodes in time and frequency (NNFFT), see \cite{ElSt98}. Moreover, the constant in the $\mathcal{O}$-notation is improved for the special case
 $\xi_k=k$ by means of the nonequispaced fast Fourier transform (NFFT), see \cite{KeKuPo09}.
\qquad \end{proof}

\begin{algorithm}[ht!]
  \caption{Evaluation of polynomials in the unit disk}
	\label{alg:XFLT}
	\begin{algorithmic}
		\Require\ \\
		$\varepsilon \in (0,1)$  \Comment{target accuracy}\\
		$N \in \N$ 	\Comment{number of sampling nodes}\\
		$\xi_1 > \xi_2 \dots > \xi_N \ge 1$ \Comment{nodes in frequency domain}\\
		$z_j\in Z=\{z\in\C:|z|\le 1\},j=1,\hdots,N$ \Comment{nodes in spatial domain} \\
		$\zb {\hat f}\in\C^N$ \Comment{Fourier coefficients}
		\Ensure\ \\
		$\mathbf{\tilde f}\in\C^N$, $\mathbf{\tilde f}\approx \zb C \zb {\hat f}$ \Comment{samples in spatial domain}
		\vspace{1ex}
		\hrule
		\vspace{1ex}
		\State $M=\left\lceil \log_2 \frac{N\log 1/\varepsilon}{\varepsilon} \right\rceil +1 $ \Comment{number of decompositions}
		\State $q=\lceil\frac{1}{2}+\log_4 1/\varepsilon\rceil$		\Comment{approximation rank, Laplace transform}
		%\State $p=\max\{10,2\log 1/\varepsilon, 2(1+\log_2 N)\}$	\Comment{approximation rank, Fourier transform}
		\vspace{1ex}
		\State $\zb {\tilde f}^{Z_0}=\zb 0$
		\vspace{1ex}
		\State $\Omega=[1,\xi_1]$
		\State $\zb {\tilde f}^{Z_M}=\zb {\tilde A}^{Y_M,\Omega} \zb {\hat f}$
		\vspace{1ex}
		\For{$\ell=1,\dots,M-1$}
		\State $\zb {\hat F}^{\Omega_{\ell}}=\left(\diag \zb {\hat f}^{\Omega_\ell}\right)\zb L^{\Omega_\ell} \in\C^{|\Omega_\ell|\times q}$
		\EndFor
		\vspace{1ex}
		\For{$m=1,\dots,M-1$}
		  \State $\Omega=\cup_{\ell>L_m} \Omega_\ell$
		  \State $\zb {\tilde f}^{Z_m}=\zb {\tilde A}^{X_m,\Omega} \zb {\hat f}^{\Omega}$
		  \vspace{1ex}
		  \State $\Omega=\cup_{\ell=\ell_m}^{L_m} \Omega_\ell$
		  \State $\zb F=\zb {\tilde A}^{X_m,\Omega}\left(
		  \zb {\hat F}^{\Omega_{\ell}} \left(\zb K^{Y_m,\Omega_{\ell}}\right)^{\top}\right)_{\ell=\ell_m,\hdots,L_m} \in\C^{|X_m|\times q}$
		  \vspace{1ex}
		  \State $\zb {\tilde f}^{Z_m} = \zb {\tilde f}^{Z_m} + \left(\mathbf{L}^{Y_m} \odot \zb 
F\right)\zb 1$
		\EndFor
			\end{algorithmic}

\end{algorithm}

We note in passing that minor improvements in computational complexity are possible by considering the
butterfly approximation scheme directly on the blocks $\zb A^{X_m,\Omega}$.
Lemma \ref{lem:AdotK} shows how to apply the Hadamard product of a low rank matrix and a matrix that allows for a fast algorithm to a vector.
The very same idea is used in \cite{ToWeOl16} for a polynomial conversion matrix being a Hadamard product of a approximately low rank Hankel matrix and a Toeplitz matrix which of course allows for fast multiplication by means of FFTs.
Moreover note that \cite{An13} suggests a fast algorithm for the multiplication with $\zb C$ when the nodes $z_j$ are close to the unit circle.
Regarding generalizations, we get a fast algorithm for the multiplication with the adjoint matrix $\zb C^*$ simply using the adjoint algorithms for the
matrices $\zb K$ and $\zb A$.
In particular, this allows to evaluate $g:\C\rightarrow\C$,
\begin{equation*}
 g(\xi)=\sum_{j=1}^{N} \hat g_j z_j^{\xi}
\end{equation*}
at nodes $\xi_k\in[1,N]$, $k=1,\hdots,N$, and for given $z_j\in Z=\{z\in\C: |z|\le 1\}$ and coefficients $\hat g_j\in\C$ efficiently.
Possible applications include the fast evaluation of certain special functions when approximated as in \cite{BeMo05} on the real line.
The most general case with kernel
\begin{equation*}
 \e^{(\xi+\ii \eta)(x+\ii y)}=\e^{\xi x}\e^{-by}\e^{\ii(\eta x+\xi y)}
\end{equation*}
allows for efficient treatment when $(\xi,\eta)$ as well as $(x,y)$ samples a smooth contour in $\C$ and are in appropriate ranges.
Then the last term leads to $2d$ sparse FFT \cite{Yi09,KuMe12} and we might apply the Hadamard product idea twice.

%%%%%%%%%%%%%%%%%%%%%%%%%%%%%%%%%%%%%%%%%%%%%%%%%%%%%%%%%%%%%%%%%%%%%%%%%%%%%%
\section{Numerical results}\label{sect:num}
	The implementation of Algorithm \ref{alg:FLT} and Algorithm
	\ref{alg:XFLT} is realized in MATLAB 2013a. We use one node of a Intel Xeon, 128GByte, 2.2GHz, Scientific Linux release 6.5 (Carbon), for all numerical experiments.
	We draw random uniformly distributed coefficients $\hat f_k$, equispaced frequencies $\xi_k=k$, $k=1,\dots,N$, and we draw random nodes $x_j \in [0,1)$, $j=1,\dots,N$ and
	random nodes $0\le y_N \le y_{N-1} \le \dots \le y_1 \le (2q-1) \log 2$, which ensures $y_j \in [0,\log{1/\varepsilon}]$.

	We consider the relative error
	\begin{equation}\label{eq:defE}
	E:=\frac{\| \zb{f} - {\zb{\tilde f}}\|_\infty }{\|\zb{\hat{f}}\|_1 },
	\end{equation}
	where $\mathbf{f}$, ${\mathbf{\tilde f}} \in \C^N $ denote the exact result and its approximation, respectively,
	see Theorem \ref{thm:approxE} and Theorem \ref{thm:approxZ}.
	Figure \ref{fig:error}\subref{fig:errora},\subref{fig:errorc},\subref{fig:errord} shows the quantity $E$ and the corresponding upper bound
	in dependence of the approximation rank $q=1,\hdots,20$ for a fixed bandwidth $N=2^{14}$.
	The error of Algorithm \ref{alg:FLT} is shown in Figure \ref{fig:error}\subref{fig:errora}.
	The daggers represent the numerical errors, the solid line the theoretical
	estimate, cf.~Lemma \ref{cor:approxE}, and the dashed line a least square fit $E \approx C_0 C^{-q}$, $C>4$.
	
% 	In Figure \ref{fig:error}\subref{fig:errorb} the local error of the butterfly Fourier transform is illustrated.
% 	The daggers show the $\ell^2$-approximation error of the trigonometric interpolation \cite{KuMe12},
% 	the circles Ying's original variant \cite{Yi09} together with the solid line representing its theoretical upper bound given in Corollary \ref{cor:FouLocalError}.
% 	The triangles are the singular values, representing the best approximation in this case, together with the lower bound of Remark \ref{rem:singvalue}
% 	(dashed line).
	
	Figure \ref{fig:error}\subref{fig:errorc},\subref{fig:errord} illustrates the error of Algorithm \ref{alg:XFLT} using the butterfly fast Fourier
	transform (BSFFT) and the nonequispaced fast Fourier transform (NFFT), respectively.
	The error of Algorithm \ref{alg:XFLT} in combination with the BSFFT can be estimated by Lemma \ref{lem:errorAdotK} and \cite[Theorem 3.1]{KuMe12}, which
	supports the choice $p=\lceil q/2\rceil +3$ for the approximation rank of the BSFFT.
	Indeed, this leads to an upper bound $C 4^{-q}$ for the total error which is however not shown in Figure \ref{fig:error}\subref{fig:errorc} since the
	theoretical constant $C$ is way too large.
	Within Algorithm \ref{alg:XFLT} in combination with the NFFT, we choose the Kaiser-Bessel window function, an approximation parameter
	$m=\lceil q/3 \rceil$, and an oversampling factor $2$ for the NFFT, see \cite{KeKuPo09}.
	This results in a theoretical upper bound of the error as shown by the solid line in Figure \ref{fig:error}\subref{fig:errord}.

\setcounter{lofdepth}{1}
		\begin{figure}[ht!]
		\subfigure[Algorithm \ref{alg:FLT}, total error.]{\label{fig:errora} \includegraphics[width=0.31\textwidth]{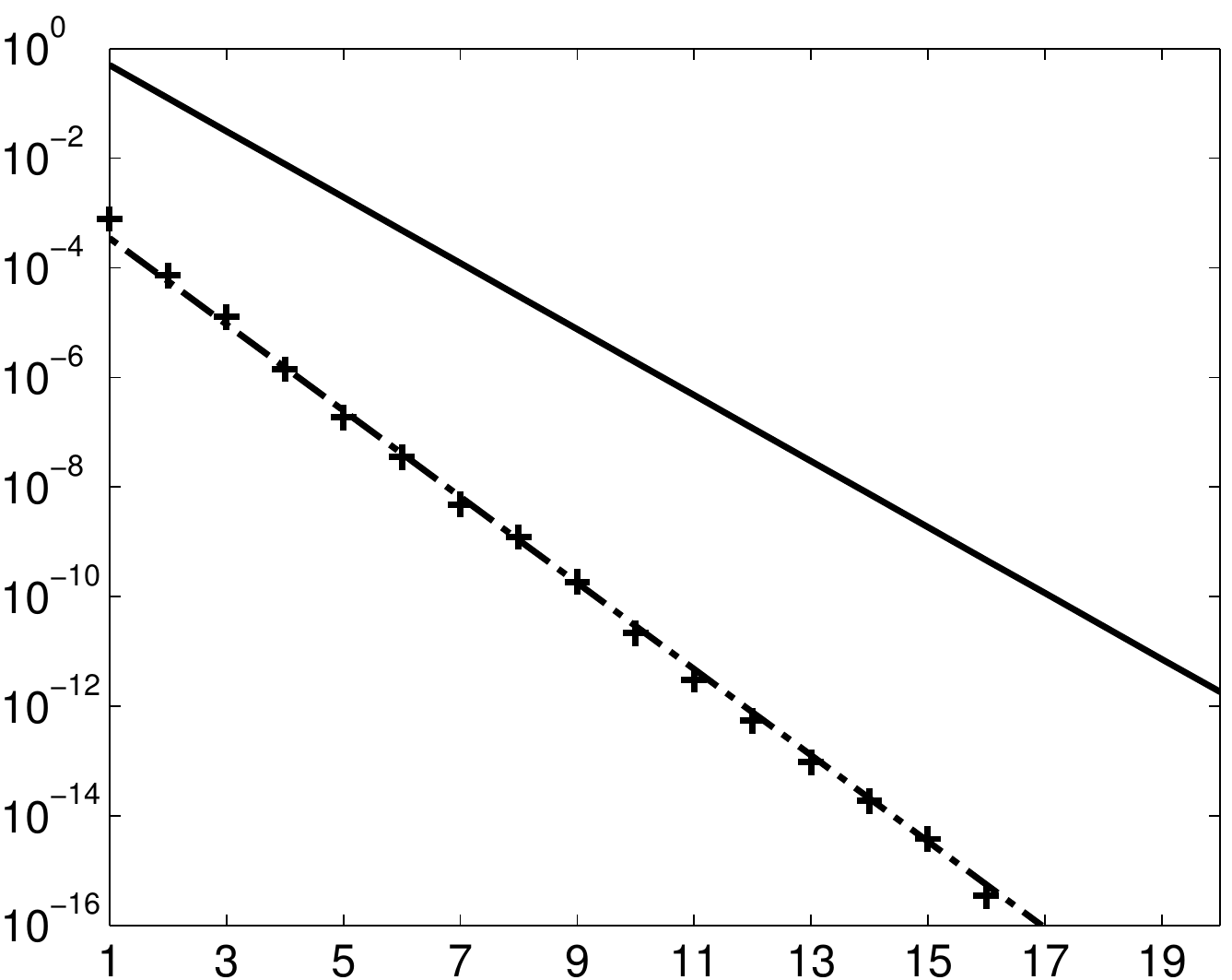} }									
		%\subfigure[Butterfly Fourier transform \cite{KuMe12}, local error.]{\label{fig:errorb} \includegraphics[width=0.45\textwidth]{images/svd_N16384}}\\
		\subfigure[Alg. \ref{alg:XFLT} (BSFFT), total error.]{\label{fig:errorc} \includegraphics[width=0.31\textwidth]{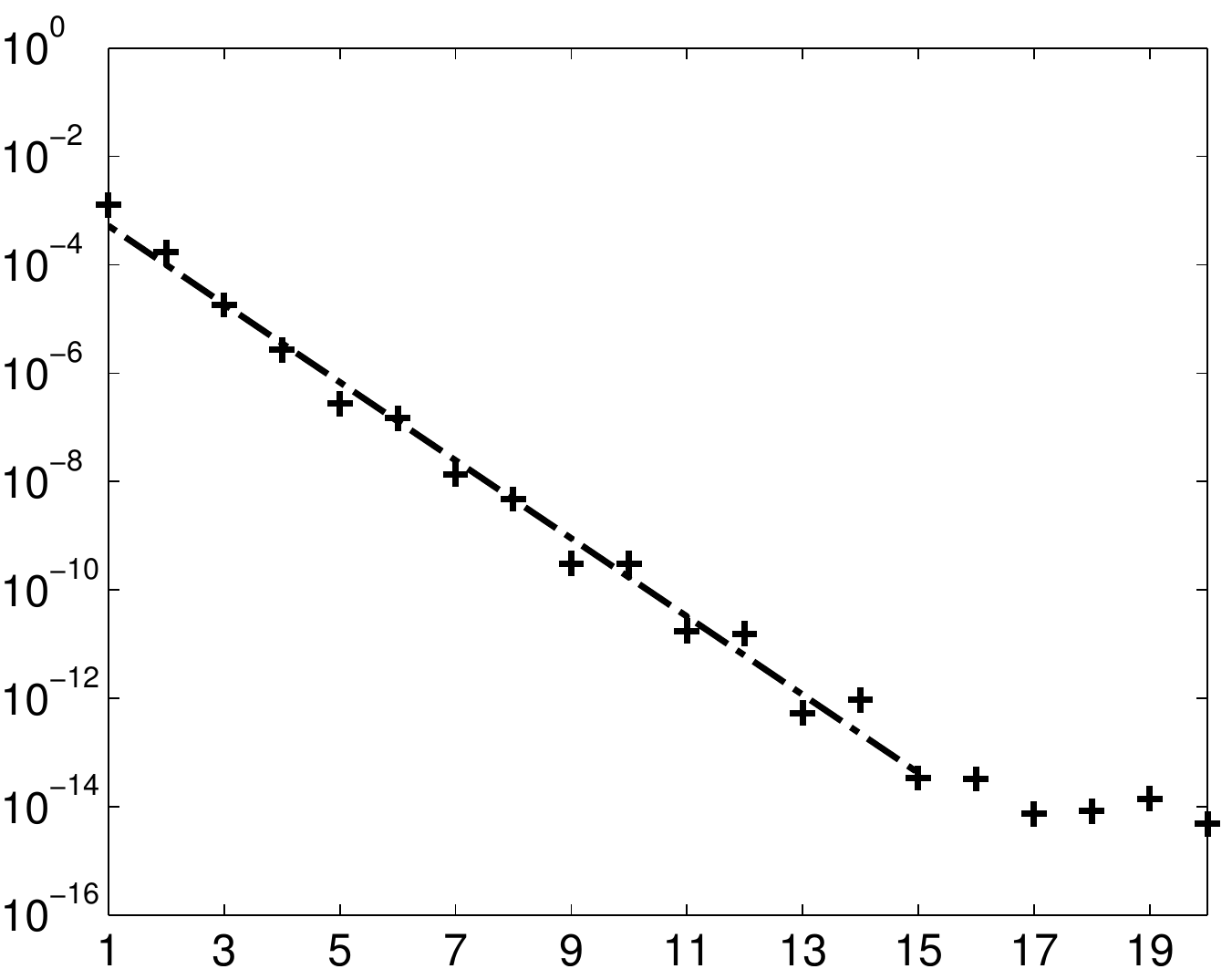}}
		\subfigure[Algorithm \ref{alg:XFLT} (NFFT), total error.]{\label{fig:errord} \includegraphics[width=0.31\textwidth]{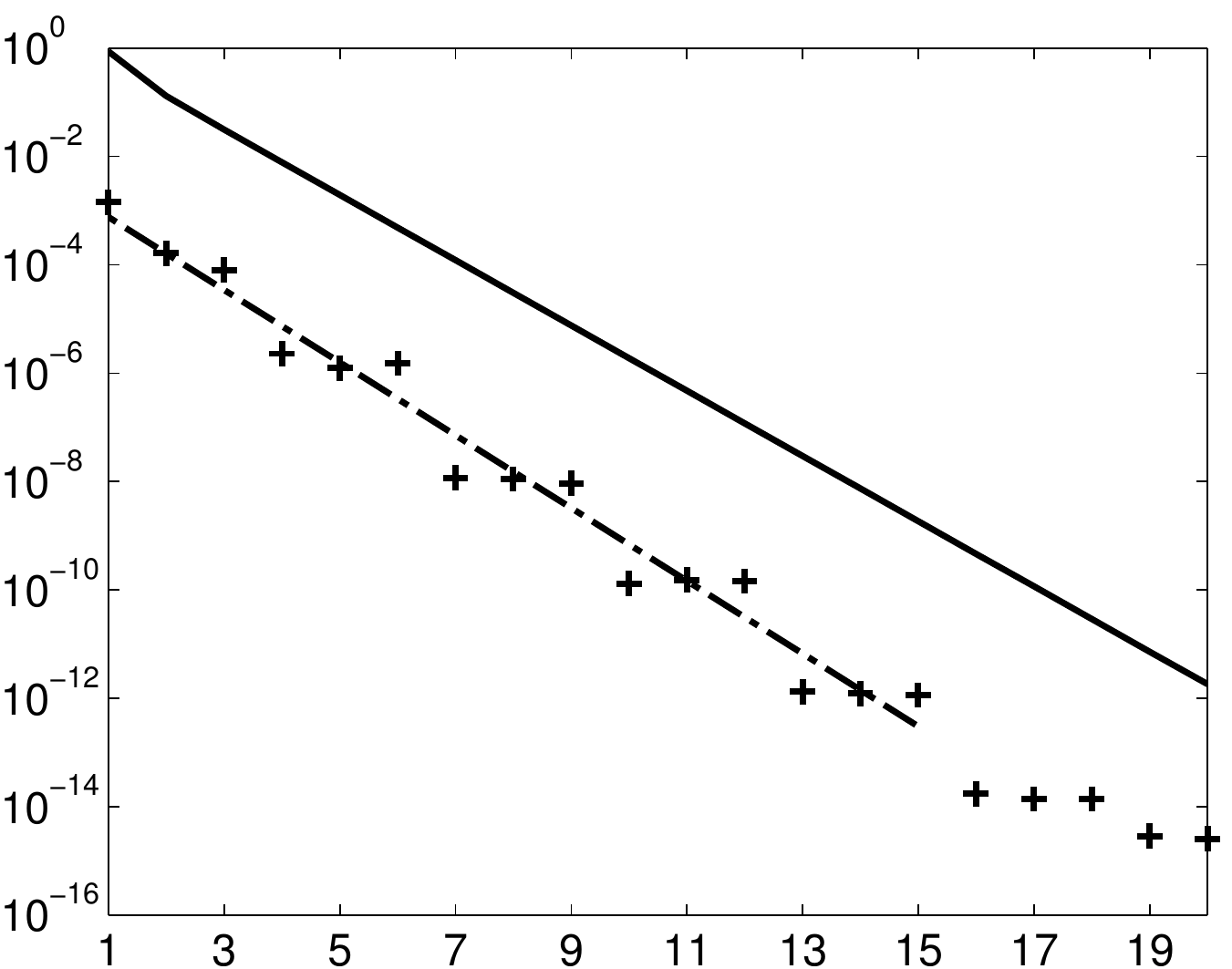}}\\
		\caption{Approximation error \eqref{eq:defE} with respect to the local expansion degree $q$.}
		\label{fig:error}
	\end{figure}
	
	In a second series of experiments, we compare the computational times, measured by the MATLAB functions \lstinline{tic} and \lstinline{toc}, of the
	naive evaluations and Algorithm \ref{alg:FLT} and \ref{alg:XFLT} with respect to increasing problem size $N$.
	Figure \ref{fig:times1D} shows the timings for the naive matrix vector multiplication using entrywise and row-wise evaluations of the matrix, both
	shown as diamonds.
	Figure \ref{fig:times1D}\subref{fig:times1Db} present the computational times for Algorithm \ref{alg:FLT}, i.e., the multiplication with the
	matrix $\zb K$. Clearly, the complexity is linear in the problem size.
	Figure \ref{fig:times1D}\subref{fig:times1Dd} show the same for Algorithm \ref{alg:XFLT} using the BSFFT (symbol $+$) and the NFFT ($*$).
	Both variants scale almost linear in $N$ but still, the constant in the BSFFT is much larger than in the NFFT case.
	\begin{figure}[ht!]
			%\subfigure[Algorithm \ref{alg:FLT}, computational times, $q=4$.]{\label{fig:times1Da} \includegraphics[width=0.45\textwidth]{images/times1D_LT_Phi_q4} }
			\subfigure[Algorithm \ref{alg:FLT}, computational times, $q=8$.]{\label{fig:times1Db} \includegraphics[width=0.45\textwidth]{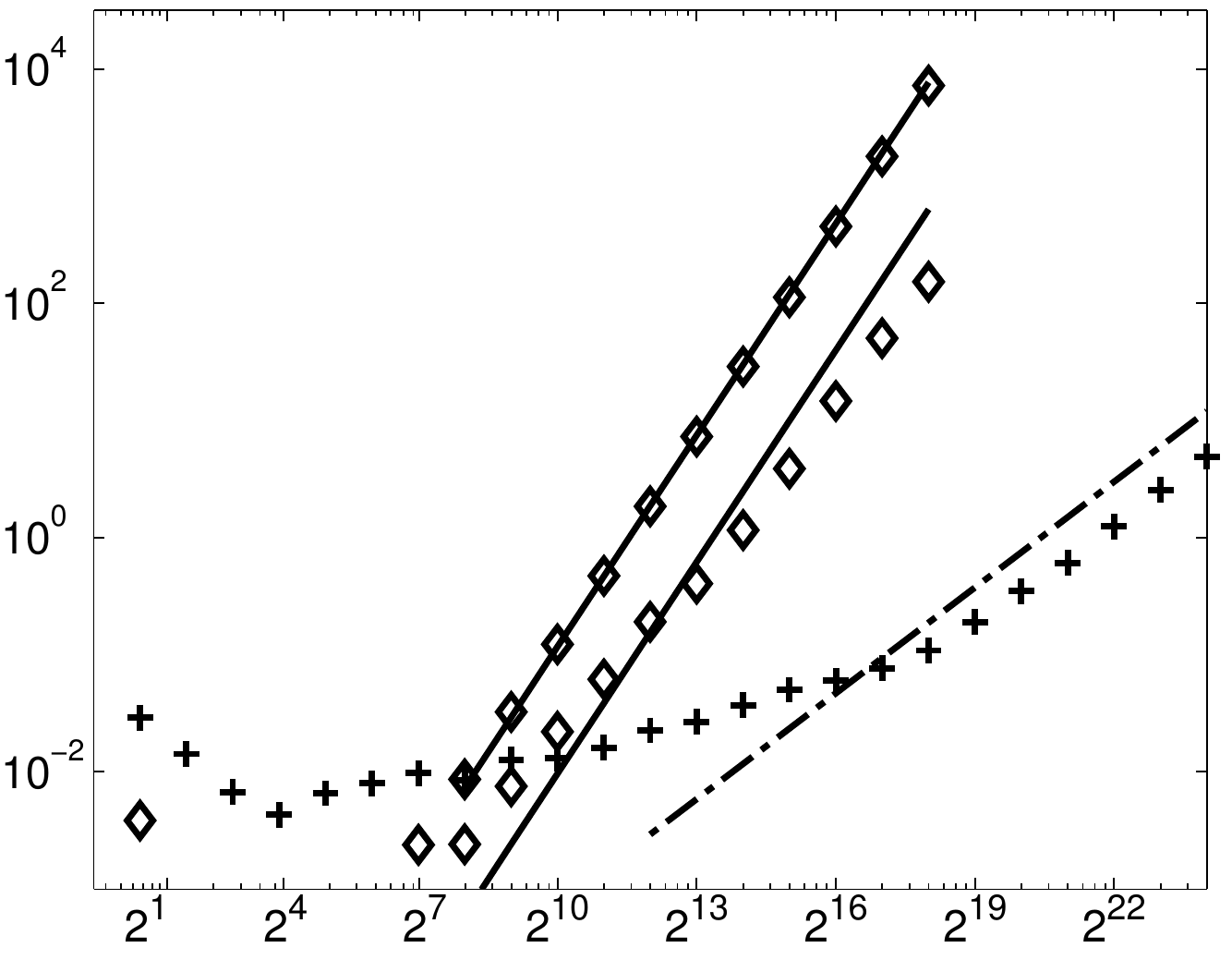} }
			%\subfigure[Algorithm \ref{alg:XFLT}, computational times, $q=4$.]{\label{fig:times1Dc} \includegraphics[width=0.45\textwidth]{images/times1D_BSFFLT_NFFLT_Phi_q4}}
			\subfigure[Algorithm \ref{alg:XFLT}, computational times, $q=8$.]{\label{fig:times1Dd} \includegraphics[width=0.45\textwidth]{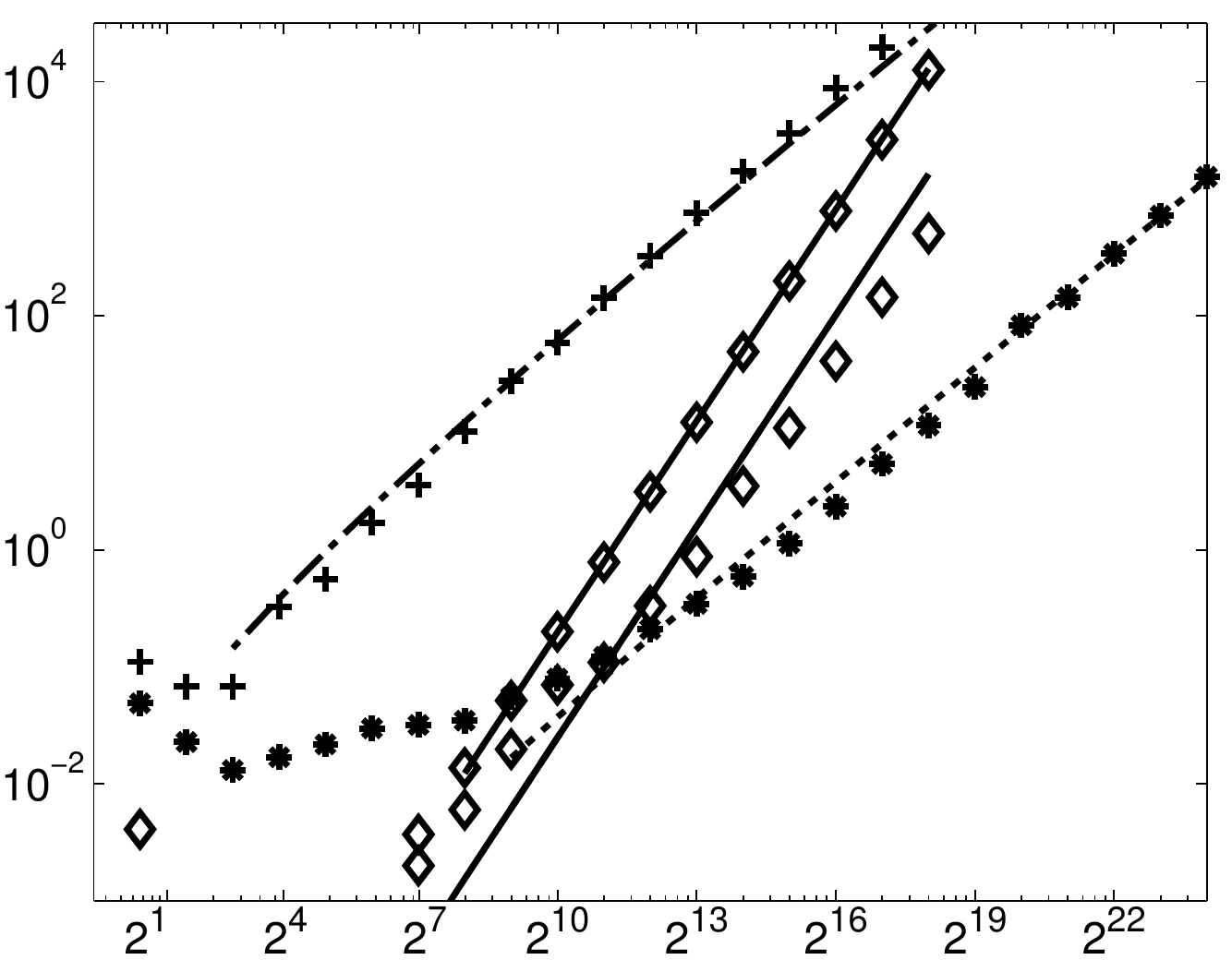}}
			\caption{Computational times for fixed approximation rank $q=8$ with respect to the problem size $N$.}
			\label{fig:times1D}
	\end{figure}

%%%%%%%%%%%%%%%%%%%%%%%%%%%%%%%%%%%%%%%%%%%%%%%%%%%%%%%%%%%%%%%%%%%%%%%%%%%%%%
\section{Summary}\label{sect:sum}

Recently, the butterfly approximation scheme and hierarchical approximations have
been proposed for the efficient computation of integral transforms with oscillatory
and with asymptotically smooth kernels.
In the first part of this paper, we summarized and slightly improved the fast discrete Laplace transform \cite{Ro88}.
We combined this Laplace transform with a generalized fast Fourier transform in a purely algebraic fashion where we used the decomposition of the Laplace transform explicitly and a small number of
generalized FFTs as black box. 
In particular, this allows for the fast evaluation of a polynomial, given by its monomial coefficients, at many nodes in the
complex unit disk. Alternatively, we might interpret this as an FFT with nonequispaced nodes in the upper half plane.
In this situation, the butterfly Fourier transform could be replaced by the nonequispaced FFT which is both asymptotically as well as
with respect to actual computation times faster.

%\subsection*
%{\bf Acknowledgment.}
%The authors %thank the referees for their valuable suggestions and
%gratefully
%acknowledge support by the DFG
%German Research Foundation
%within the project KU
%2557/1-1 and by the Helmholtz Association within the young investigator group
%VH-NG-526.

%%%%%%%%%%%%%%%%%%%%%%%%%%%%%%%%%%%%%%%%%%%%%%%%%%%%%%%%%%%%%%%%%%%%%%%%%%%%%%
\begin{appendix}

\appendix\section*{} \label{App:AppendixA}
The following results for the exponential kernel are a simplification and minor improvement of \cite{Ro88} and keep the paper self contained.
\begin{lemma}{\cite[eq. (36)]{Ro88}}\label{cor:approxE}
 Let $q\in\N$, $A,B\subset[0,\infty)$ be admissible, and $\kappa(y,\xi)=\e^{-\xi y}$, then
 \begin{equation*}
  \left\|\kappa - \mathcal{I}_q^{A\times B}\kappa\right\|_{C(A\times B)} \le 2^{1-2q}.
 \end{equation*}
\end{lemma}
\begin{proof}
 For $y>0$, we have the necessary condition for a local maximum of $|\partial_y^q \kappa(y,\xi)|$, i.e., 
\begin{equation*}
 \left|\partial_{\xi} \partial_y^q \kappa(y,\xi)\right| = \xi^{q-1} \e^{-\xi y} \left|q-\xi y\right| =0
\end{equation*}
if and only if $\xi=q/y$. Using Stirling's approximation, we conclude the globally valid bound
\begin{equation*}
 \left|\partial_y^q \kappa(y,\xi)\right| = \xi^q \e^{-\xi y} \le \frac{q^q}{\e^q y^q} \le \frac{1}{\sqrt{2\pi q}} q! y^{-q},
\end{equation*}
 i.e., the exponential kernel is asymptotically smooth with constants $C=1/\sqrt{2\pi}$, $\mu=1$, $s=0$, and $\nu=-1/2$.
 The result follows from Theorem \ref{thm:LapLocalError} since $(2+\frac{2}{\pi}\log q)/\sqrt{2\pi q}\le 1$.
\qquad \end{proof}

\begin{lemma}{\cite[Sect.~4]{Ro88}}\label{lem:approxE}
 Let $\varepsilon,y_1,\xi_1>0$ be given, use the notation of Definition \ref{def:M} and Lemma \ref{cor:approxE} and
 set $q:=\lceil\frac{1}{2}+\log_4 1/\varepsilon\rceil$,
 \begin{equation}\label{eq:sumidx}
   \ell_m  := \max(1,\lfloor \log_2(y_1 \xi_1) - m - \log_2(\log 1/\varepsilon)\rfloor + 1), \qquad
   L_m  := M-m
 \end{equation}
for all $m=1,\dots,M-1$,
then
 \begin{romannum}
  \item $y \in Y_M$ and $\xi \in \Omega$ (and analogously $y \in Y$ and $\xi \in \Omega_M$) implies
      $1-\e^{-y \xi} \le \varepsilon$,
  \item $y \in Y_m$, $m=1,\hdots,M-1$, and $\xi \in \Omega_l$, $l< \ell_m$, implies
	$\e^{-y \xi} \leq \varepsilon$,
  \item $y \in Y_m$, $m=1,\hdots,M-1$, and $\xi \in \Omega_l$,$\ell_m \le l \le L_m$, implies
        $\left|\e^{-y \xi}-\mathcal{I}_q^{Y_m\times \Omega_l}\kappa(y,\xi)\right| \le \varepsilon$,
  \item $y \in Y_m$, $m=1,\hdots,M-1$, and $\xi \in \Omega_l$, $l> L_m$, implies
        $1-\e^{-y \xi} \le \varepsilon$.
 \end{romannum}
\end{lemma}
{\em Proof}.
 The individual estimates can be proven as follows.
 At first, let $\xi \in [0,\xi_1]$  and $y \in Y_M=[0,y_1/2^{M-1}]$.
 Using $M \geq \log_2 \frac{y_1\xi_1}{\varepsilon} +1$, we obtain $0 \leq \xi y \leq y_1 \xi_1/2^{M-1}\leq \varepsilon$
 and finally case i) since
 \begin{equation*}
    1 \geq \e^{-y\xi} \geq \e^{-\varepsilon}=\sum_{k=0}^{\infty} \frac{(-\varepsilon)^k}{k!}\geq 1-\varepsilon.
 \end{equation*}	

 Now let $y\in Y_m$, $\xi \in \Omega_\ell$. The condition $\ell \leq \lfloor \log_2(y_1 \xi_1) - m - \log_2(\log 1/ \varepsilon) \rfloor$
 implies 
  \begin{equation*}
   \e^{- \frac{y_1 \xi_1}{2^m 2^\ell }} \leq \varepsilon
  \end{equation*}
 and due to $y\geq \frac{y_1}{2^m}$, $\xi\geq \frac{\xi_1}{2^\ell}$ assertion ii).
 The third result follows from Lemma \ref{cor:approxE} since the intervals $Y_m$, $\Omega_\ell$, $\ell,m =1,\dots,M-1$, are admissible.
 
 Finally, we have $y\in Y_m$, $\xi\in\Omega_{\ell}$, $\ell-1 \geq \lceil \log_2(\xi_1 y_1) + \log_2(1/\varepsilon) \rceil - (m-1)$ and thus
 \begin{equation*}
  1-\e^{-y\xi}\le 1-\e^{-\frac{y_1}{2^{m-1}} \frac{\xi_1}{2^{\ell-1}}} \le 1-\e^{-\varepsilon} \le \varepsilon. \eqno\endproof
 \end{equation*}

\begin{figure}[H]
	\centering
	\includegraphics[width=0.8\textwidth]{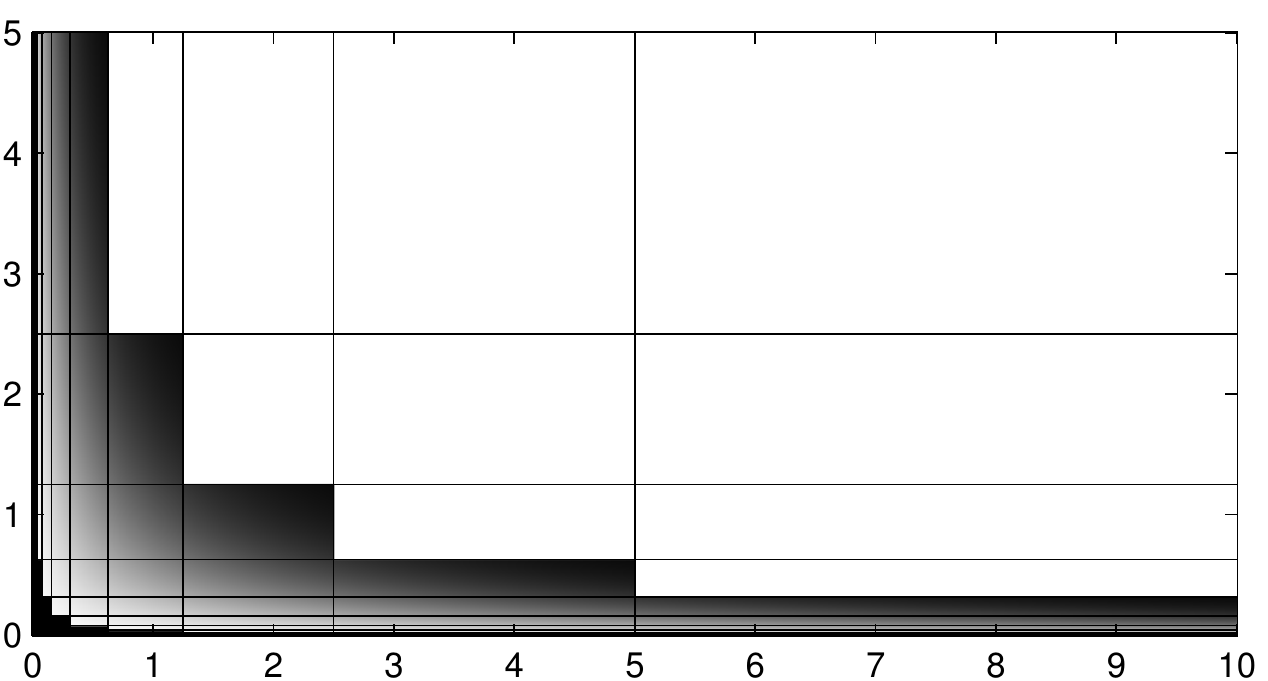}
	\caption{Kernel function and dyadic decomposition, the cases i),ii)) and iv) are shown in black and white, respectively.}
	\label{fig:approx}
\end{figure}

\begin{theorem}{\cite[Sect.~7]{Ro88}}\label{thm:approxE}
 Let $N\in\N$, $\varepsilon>0$, $\mathbf{\hat f}\in\C^N$, 
 $y_1>y_2>\hdots>y_N>0$, $\xi_1>\xi_2>\hdots>\xi_N>0$,  and
 $\kappa(y,\xi):=\e^{-\xi y}$ be given.
  %$\mathbf{f}:=\mathbf{K}\mathbf{\hat f}$ with
 %\begin{equation}\label{eq:matK}
 % \mathbf{K}=(\e^{-y_j \xi_{\ell}})_{j,\ell=1,\hdots,N}
 %\end{equation}
 %be given.
 Set $M:=\left\lceil \log_2 \frac{y_1 \xi_1}{\varepsilon} \right\rceil +1$,
 $q:=\lceil\frac{1}{2}+\log_4 1/\varepsilon\rceil$, and
 $\mathbf{\tilde f}=(\tilde{f}(y_i))_{i=1,\hdots,N}$ for the function $\tilde f:Y\rightarrow \C$,
 \begin{equation}\label{eq:tildef}
  \tilde f (y) = \begin{cases}
    \sum_{j=1}^{N} \hat f_j & y \in Y_M,\\
    \sum_{\ell=\ell_m}^{L_m} \sum_{\xi_j \in \Omega_\ell} \hat f_j \mathcal{I}^{Y_m \times \Omega_\ell} \kappa(y,\xi_j)
     + \sum_{\ell > L_m} \sum_{\xi_j \in \Omega_\ell} \hat f_j & y\in Y_m,\; 1\le m<M.%m=1,\hdots,M-1.
                 \end{cases}
 \end{equation}
 Then the error estimate
 $%\begin{equation*}
  \|\mathbf{f}-\mathbf{\tilde f}\|_{\infty} \le \varepsilon \|\mathbf{\hat f}\|_1
 $ %\end{equation*} 
 holds true.
 %and Algorithm \ref{alg:FLT} computes this approximation in
 %\begin{equation*}
%   \mathcal{O}\left(N\log\frac{1}{\varepsilon}+\log^3\frac{1}{\varepsilon}\log\frac{y_1\xi_1}{\varepsilon}\right)
% \end{equation*}
% floating point operations.
\end{theorem}
\begin{proof}
 We start with the error estimate. For $y\in Y_M$, Lemma \ref{lem:approxE}(i) implies
 \begin{equation*}
  \left|f(y) - \tilde f(y)\right| \le \sum_{k=1}^{N} |\hat f_k| |\e^{-\xi_k y}-1| \le \varepsilon \sum_{k=1}^{N} |\hat f_k|.
 \end{equation*}
 Now let $m=1,\hdots,M-1$, $y \in Y_m$, and partition the function $f$ in three parts
 \begin{equation*}
  f(y) =  \left(\sum_{\ell < \ell_m}
        + \sum_{\ell=\ell_m}^{L_m}
        + \sum_{\ell > L_m}\right)\sum_{\xi_j \in \Omega_\ell} \hat f_j \kappa(y,\xi_j).
 \end{equation*}
 The desired result follows by the application of Lemma \ref{lem:approxE}ii)-iv) and the approximation of the
 kernel $\kappa$ by zero, by interpolation, or by one, respectively. \qquad\end{proof}

\end{appendix} 

\bibliographystyle{abbrv}
\bibliography{../references}

\begin{thebibliography}{10}

\bibitem{An13}
F.~Andersson.
\newblock Algorithms for unequally spaced fast {L}aplace transforms.
\newblock {\em Appl. Comput. Harmon. Anal.}, 35:419 -- 432, 2013.

\bibitem{AyChSoCu03}
A.~A. Ayd{\i}ner, W.~C. Chew, J.~Song, and T.~J. Cui.
\newblock A sparse data fast {F}ourier transform ({SDFFT}).
\newblock {\em IEEE Trans. Antennas and Propagation}, 51(11):3161--3170, 2003.

\bibitem{Be00}
M.~Bebendorf.
\newblock Approximation of boundary element matrices.
\newblock {\em Numer. Math.}, 86:565 -- 589, 2000.

\bibitem{Be08}
M.~Bebendorf.
\newblock {\em {H}ierarchical {M}atrices}, volume~63 of {\em {L}ecture {N}otes
  in {C}omputational {S}cience and {E}ngineering}.
\newblock Springer-Verlag, 2008.

\bibitem{BeTr04}
J.-P. Berrut and L.~N. Trefethen.
\newblock Barycentric {L}agrange interpolation.
\newblock {\em \textrm{SIAM} Rev.}, 46:501 -- 517, 2004.

\bibitem{Bey95}
G.~Beylkin.
\newblock On the fast {F}ourier transform of functions with singularities.
\newblock {\em Appl. Comput. Harmon. Anal.}, 2(4):363--381, 1995.

\bibitem{BeMo05}
G.~Beylkin and L.~Monz{\'o}n.
\newblock On approximations of functions by exponential sums.
\newblock {\em Appl. Comput. Harmon. Anal.}, 19:17 -- 48, 2005.

\bibitem{CaDeYi09}
E.~Cand{\`e}s, L.~Demanet, and L.~Ying.
\newblock A fast butterfly algorithm for the computation of {F}ourier integral
  operators.
\newblock {\em Multiscale Model. Simul.}, 7(4):1727--1750, 2009.

\bibitem{CoTu65}
J.~W. Cooley and J.~W. Tukey.
\newblock An algorithm for the machine calculation of complex {F}ourier series.
\newblock {\em Math. Comput.}, 19(90):297--301, 1965.

\bibitem{DuVe90}
P.~Duhamel and M.~Vetterli.
\newblock Fast {F}ourier transforms: a tutorial review and a state of the art.
\newblock {\em Signal Process.}, 19(4):259--299, 1990.

\bibitem{DuRo95}
A.~Dutt and V.~Rokhlin.
\newblock Fast {F}ourier transforms for nonequispaced data. {II}.
\newblock {\em Appl. Comput. Harmon. Anal.}, 2(1):85--100, 1995.

\bibitem{ElSt98}
B.~Elbel and G.~Steidl.
\newblock Fast {F}ourier transform for nonequispaced data.
\newblock In C.~K. Chui and L.~L. Schumaker, editors, {\em Approximation Theory
  IX}, pages 39 -- 46, Nashville, 1998. Vanderbilt University Press.

\bibitem{FFTW05}
M.~Frigo and S.~G. Johnson.
\newblock The design and implementation of {FFTW3}.
\newblock {\em Proceedings of the {\rm IEEE}}, 93(2):216 -- 231, 2005.

\bibitem{GrHa03}
L.~Grasedyck and W.~Hackbusch.
\newblock Construction and arithmetics of hierarchical matrices.
\newblock {\em Computing}, 70:295 -- 334, 2003.

\bibitem{GrRo87}
L.~Greengard and V.~Rokhlin.
\newblock A fast algorithm for particle simulations.
\newblock {\em J. Comput. Phys.}, 73:325 -- 348, 1987.

\bibitem{Ha99}
W.~Hackbusch.
\newblock A sparse matrix arithmetic based on {$\mathcal{H}$}--matrices, {P}art
  {I}: introduction to {$\mathcal{H}$}--matrices.
\newblock {\em Computing}, 62:89 -- 108, 1999.

\bibitem{Ha09}
W.~Hackbusch.
\newblock {\em {H}ierarchische {M}atrizen. {A}lgorithmen und {A}nalysis}.
\newblock Springer-Verlag, 2009.

\bibitem{KeKuPo09}
J.~Keiner, S.~Kunis, and D.~Potts.
\newblock Using {NFFT} 3 -- a software library for various nonequispaced fast
  {F}ourier transforms.
\newblock {\em ACM Trans. Math. Software}, 36(4):Art. 19, 30, 2009.

\bibitem{KuMe12}
S.~Kunis and I.~Melzer.
\newblock A stable and accurate butterfly sparse {F}ourier transform.
\newblock {\em SIAM J. Numer. Anal.}, 50(3):1777--1800, 2012.

\bibitem{MiBo96}
E.~Michielssen and A.~Boag.
\newblock A multilevel matrix decomposition algorithm for analyzing scattering
  from large structures.
\newblock {\em IEEE Trans. Antennas and Propagation}, 44(8):1086 --1093, 1996.

\bibitem{PoDeMaYi14}
J.~Poulson, L.~Demanet, N.~Maxwell, and L.~Ying.
\newblock A parallel butterfly algorithm.
\newblock {\em SIAM J. Sci. Comput.}, 36(1):C49--C65, 2014.

\bibitem{Ro88}
V.~Rokhlin.
\newblock A fast algorithm for the discrete {L}aplace transformation.
\newblock {\em J. Complexity}, 4(1):12--32, 1988.

\bibitem{St98}
G.~Steidl.
\newblock A note on fast {F}ourier transforms for nonequispaced grids.
\newblock {\em Adv. Comput. Math.}, 9(3-4):337--352, 1998.

\bibitem{St92}
J.~Strain.
\newblock A fast {L}aplace transform based on {L}aguerre functions.
\newblock {\em {Math. Comp.}}, 58(197):275--283, 1992.

\bibitem{SuPi01}
X.~Sun and N.~P. Pitsianis.
\newblock A matrix version of the fast multipole method.
\newblock {\em \textrm{SIAM} Rev.}, 43:289 -- 300, 2001.

\bibitem{ToWeOl16}
A.~Townsend, M.~Webb, and S.~Olver.
\newblock Fast polynomial transforms based on {T}oeplitz and {H}ankel matrices.
\newblock {\em ArXiv e-prints}, 2016.

\bibitem{Yi09}
L.~Ying.
\newblock Sparse {F}ourier transform via butterfly algorithm.
\newblock {\em SIAM J. Sci. Comput.}, 31(3):1678--1694, 2009.

\bibitem{YiBiZo04}
L.~Ying, G.~Biros, and D.~Zorin.
\newblock A kernel-independent adaptive fast multipole method in two and three
  dimensions.
\newblock {\em J. Comput. Phys.}, 196:591 -- 626, 2004.

\end{thebibliography}

\end{document}